\documentclass[a4paper,11pt]{article}

\usepackage{amsmath}
\usepackage{amsfonts}
\usepackage{amssymb}
\usepackage{amsthm}
\usepackage{graphicx}
\usepackage{subfigure}
\usepackage{enumerate}
\usepackage{color}
\usepackage{hyperref}

\newcommand{\E}{\ensuremath{\mathcal{E}}}

\newcommand{\R}{\mathbb{R}}
\newcommand{\Rn}{\mathbb{R}^d}
\newcommand{\Rd}{\mathbb{R}^d}

\newcommand{\bv}{\mathbf{v}}

\newcommand{\tbj}{\widetilde{\mathbf{J}}}
\newcommand{\bpsi}{\boldsymbol{\psi}}

\newcommand{\N}{\ensuremath{\mathbb{N}}}

\newcommand{\ol}[1]{\overline{#1}}

\newcommand{\eps}{\ensuremath{\varepsilon}}
\newcommand{\ep}{\ensuremath{\varepsilon}}
\newcommand{\tc}{c}

\newcommand{\LL}{\mathcal{L}}

\newcommand{\duality}[2]{\left\langle #1, #2 \right\rangle}

\newcommand{\no}{\mathbf{n}}
\newcommand{\ve}{\mathbf{v}}
\newcommand{\we}{\mathbf{w}}

\newcommand{\Div}{\operatorname{div}}

\newtheorem{definition}{Definition}[section]
\newtheorem{remark}[definition]{Remark}

\newtheorem{lemma}[definition]{Lemma}
\newtheorem{theorem}[definition]{Theorem}

\newtheorem{examples}[definition]{Examples}

\numberwithin{equation}{section}  

\begin{document}
\begin{titlepage}
\title{Existence and Nonlocal-to-Local Convergence for Singular, Anisotropic Nonlocal Cahn-Hilliard Equations}
\author{Helmut Abels\footnote{Fakult\"at f\"ur Mathematik,  
Universit\"at Regensburg,
93040 Regensburg,
Germany, e-mail: {\sf helmut.abels@mathematik.uni-regensburg.de}}\ \ and 
Yutaka Terasawa\footnote{Graduate School of Mathematics, Nagoya University, Furocho Chikusaku, Nagoya, 464-8602, Japan, e-mail: {\sf yutaka@math.nagoya-u.ac.jp}}
}
\date{\today\\[2ex]
\emph{Dedicated to Robert Denk on the occassion of his 60th birthday.}}
\end{titlepage}
\maketitle
\begin{abstract}
We study the nonlocal-to-local convergence for a nonlocal Cahn-Hilliard equation with anisotropic and singular kernels. In particular, we show convergence of weak solutions of the nonlocal Cahn-Hilliard equation to weak solutions of a corresponding anisotropic Cahn-Hilliard equation for suitable subsequences. Moreover, we show existence of weak solutions for the nonlocal equation under a condition, which guarantees existence of weak solutions for suitably localized or singular kernels.   
\end{abstract}

\noindent{\bf AMS-Classification:} 
Primary: 45K05; 
Secondary: 35B40, 
35K61 


\section{Introduction} \label{intro}
The Cahn-Hilliard equation was introduced by Cahn and Hilliard \cite{cahn1958free} in 1958 to describe
phase separation of binary alloys. It is one of the basic so-called phase field model. Variants of it are used in models with two or multiple components such as two-phase flows of fluids, tumor growth and many other situations, cf.\ e.g.\ \cite{MiranvilleCH}. 
The basic Cahn-Hilliard equation with standard boundary conditions yields the system
\begin{alignat}{2}
    \partial_t c &= \Div (m(c)\nabla \mu_\eps)\ &\qquad&\text{in}\ \Omega\times (0,T), \label{eq:CH1}\\
    \mu &=  - \Delta c + f'(c)&\qquad&\text{in}\ \Omega\times (0,T), \label{eq:CH2}\\
    \partial_{\no} \mu|_{\partial\Omega}  &= \partial_\no c|_{\partial \Omega} = 0 &\qquad&\text{on}\ \partial\Omega\times (0,T), \label{eq:CH3}\\
    c|_{t=0}&= c_0 &&\text{in}\ \Omega. \label{eq:CH4}
\end{alignat}
Here $ \Omega $ is a bounded domain with $C^2$-boundary and exterior normal $\no$, $ c $ is an order parameter related to the concentrations of the components (e.g., the concentration difference), $ \mu $ is a chemical potential, $f$ is a suitable so-called double-well potential, which will be specified later.  Finally, $ m \colon \R\to (0,\infty)$ is a mobility coefficient, which is assumed to be sufficiently smooth. Notice that 
the equations in \eqref{eq:CH3} are Neumann boundary conditions for the chemical potential $ \mu $ and the order parameter $c$. These describe a no-flux and a ninety degree boundary condition for the diffuse interface, respectively. 

In the case that $f$ is a suitable fourth order polynomial Elliott and Zheng~\cite{ElliottZhengCH} proved existence of a regular solutions for all times in $L^2$-Sobolev spaces. A disadvantage of it is that smooth $f$ do not ensure that $c(x,t)\in [-1,1]$ almost everywhere, where $[-1,1]$ is the physical reasonable range for $c$. A counter-example and further references can be found in \cite{MiranvilleCH}.
One way out is to use suitable singular $f$. An important example is given by $f\colon [-1,1]\to \R$ with
\begin{equation*} 
f(s) = \frac{\theta}{2} ((1+s) \log(1+s) + (1-s) \log (1-s)) - \frac{\theta_c}{2}s^2 \qquad \text{for all }s\in [-1,1]
\end{equation*}  
for some $0 < \theta_c < \theta$, which was introduced in \cite{cahn1958free}. In this case well-posedness of the equation was first shown in the preprint of Elliott and Luckhaus~\cite{ElliottLuckhaus}. Alternative proofs were given by Debussche and Dettori~\cite{DebusscheDettori}, A. and Wilke~\cite{AsymptoticCH}, and Kenmochi et al.~\cite{KenmochiEtAl} in a more general context.  
In the following we will assume that $f\colon [-1,1]\to \R$ is continuous, twice continuously differentiable in $(-1,1)$, and $\inf_{s\in (-1,1)} f''(s)>-\infty$.


The nonlocal Cahn-Hilliard equation was introduced by Giacomin and Lebowitz~\cite{GiacominLebowitz} as a limit of interacting particle systems and leads to the system
\begin{alignat}{2}
    \partial_t c &= \Div (m(c)\nabla \mu)\ &\qquad&\text{in}\ \Omega\times (0,T), \label{eq:nlCH1}\\
    \mu &= \LL c + f'(c)&\qquad&\text{in}\ \Omega\times (0,T), \label{eq:nlCH2}\\
    \partial_{\no} \mu|_{\partial\Omega}  &= 0 &\qquad&\text{on}\ \partial\Omega\times (0,T), \label{eq:nlCH3}\\
    c|_{t=0}&= c_0 &&\text{in}\ \Omega. \label{eq:nlCH4}
\end{alignat}
Here $ \mathcal{L}$ is a nonlocal linear operator defined as 
\begin{align}\label{eq:defnL}
    \LL u(x) 
    &=\lim_{\delta \to 0} \int_{\Omega\setminus B_\delta(x)} (u(x)-u(y))k(x-y)dy
\end{align}
for suitable $u\colon \Omega \to \R$, where  $k \colon \Rd\setminus\{0\}\to [0,\infty)$ is nonnegative and satisfies $ k(z) = k(-z)$ for all $z\in \R^d\setminus \{0\}$.
In the case that $k$ is a suitable integrable kernel, Gajewski and Zacharias~\cite{GajewskiZacharias} have shown existence and uniqueness of suitable (weak) solutions and studied their asymptotic behavior in the case that $f$ is a suitable regular potential. We refer to Giorgini~\cite{GiorginiSepeartionNonlocalCH} and Poiatti~\cite{Poiatti3DStrictSeparationNonlocalCH} on recent results concerning the analysis in the case of a singular free energy and to Lv and Wu~\cite{LvWuNonlocalBCs} for analytic result in the case of nonlocal dynamic boundary conditions. In these contributions further references can be found.

We note that in the case
$$ 
k(z) = C(d, \alpha) \frac{1}{|z|^{d + \alpha}} \qquad \text{for } z\in \R^n\setminus \{0\}, 
$$ 
where $ \alpha \in (1, 2) $ and $ C(d, \alpha) = \frac{\alpha 4^{\alpha}}{\pi^\frac{d}{2}} \frac{\Gamma(\alpha+\frac{d}{2})}{\Gamma(1-\alpha)} $, the operator $ \mathcal{L} $ is a regional fractional Laplacian. For such a class of operators existence and uniqueness of weak solutions was shown by A., Bosia, and Grasselli~\cite{AbelsBosiaGrasselli}.  Regional fractional Laplacians are generators of censored stable processes and were introduced by Bogdan, Burdzy, and Chen~\cite{BogdanBurdzyChen}. They have been studied both in a probabilistic and an analytic context. We refer to Daoud and Laamri~\cite{DaoudLaamri} for a survey of various types of fractional Laplacians. We note that the regional fractional Laplacian defined above is called restricted fractional Laplacian in  \cite{DaoudLaamri}. A recent interesting development for regional fractional Laplacians is an optimal boundary regularity for an elliptic Poisson equation involving regional fractional Laplacian in Fall~\cite{FallBoundaryRegularity}.
In \cite{AbelsBosiaGrasselli} operators, which have non-translation-invariant kernels,
are also treated and boundary regularity of solutions is discussed. Finally, we refer to Davoli, Gavioli, and Lombardini~\cite{DavoliGavioliLombardiniCHSingularKernels} for a recent result on existence of solutions for a large class of singular kernels.

Recently, the nonlocal-to-local convergence of Cahn-Hilliard equations has been much studied. This means that solutions to the nonlocal Cahn-Hilliard equation converges to those of the
local Cahn-Hillard equation when the kernel $k$ concentrates at the origin suitably. This was first studied by
Davoli, Ranetbauer, and Scarpa~\cite{DavoliEtAlNonlocalLocalCHPeriodic} in a periodic setting. Nonlocal-to-local convergence in
a bounded domain with Neumann boundary condition was studied by Davoli, Scarpa and Trussardi~\cite{DavoliEtAlNonlocalLocalCHViscosity}.
In Davoli, Scarpa, and Trussardi~\cite{DavoliEtAlW11}, the kernel is not assumed to be $ W^{1}_{1}$ as in the previous literature. Because of that, they had to add a viscosity term to the chemical potential to get the existence of solutions to the nonlocal equation. These results rely on results by Bourgain, Brezis, and Mironescu~\cite{BourgainBrezisMironescu, BourgainBrezisMironescu2} and Ponce~\cite{PonceNonlocalPoincare,PonceNonlocalLocalConvergence}.
Their results give a new characterization of Sobolev spaces using nonlocal operator, nonlocal Poincar\'{e} inequalities and $\Gamma$-convergence of nonlocal energy functional. Moreover, nonlocal Poincar\'{e} inequalities give rise to Rellich type compactness lemmas. We note that in A. and Hurm~\cite{AbelsHurm} rates of convergence have been shown under suitable assumptions.

Nonlocal-to-local convergence results are further extended to various coupled systems, cf.\ e.g., A. and T.~\cite{AbelsTerasawaNonlocalToLocal}, Kurima~\cite{Kurima}, Colli, Kurima, and Scarpa~\cite{ColliKurimaScarpa}, Hurm, Knopf, and Poiatti~\cite{HurmKnopfPoiatti}, and Hurm and Moser~\cite{HurmMoser}. 

In these previous studies of nonlocal-to-local convergence for the Cahn-Hilliard equation, the kernels for the nonlocal operator are assumed to be radially symmetric and in $W^{1, 1}$.
On the other hand, the kernels for the nonlocal Cahn-Hilliard equation derived by Giacomin and Lebowitz ~\cite{GiacominLebowitz} are even,
i.e., $ k(z) = k(-z) $ for all $ z \in \Rd $ and not necessarily radially symmetric.  Many subsequent works such as Gal, Giorgini and Grasselli~\cite{GalGiorginiGrasselli}, Frigeri and Grasselli~\cite{FrigeriGrasselli}, Colli, Frigeri and Grasselli~\cite{ColliFrigeriGrasselli}, Frigeri~\cite{FrigeriNonlocalAGG}, to name a few, work on the equation with an even $W^{1,1}$ kernel, which is not necessarily radially symmetric. Nonlocal-to-local convergence of solutions to these equations is not covered in the previous studies and existence only partly. Since the kernels associated with regional fractional Laplacians is not integrable; in particular, the nonlocal Cahn-Hilliard equation
treated in A., Bosia and Grasselli~\cite{AbelsBosiaGrasselli} was also not covered. 

In this contribution, we would like to fill these gaps in previous study of nonlocal-to-local convergence of solutions to the Cahn-Hilliard equation. We aim to show nonlocal-to-local convergence of solutions to the nonlocal Cahn-Hilliard equation in the case of an even kernel, which is not necessarily radially symmetric and is singular, i.e., not integrable.
We also show that the existence of solutions for nonlocal Cahn-Hilliard equations when it is close to local ones in a suitable sense. As a limit, we obtain a certain anisotropic Cahn-Hilliard equation, where the Laplacian in the chemical potential is replaced by an elliptic differential operator of second order with constant coefficients.
Note that the well-posedness of more general anisotropic Cahn-Hillard equations was studied e.g.\ by Garcke, Knopf, and Wittmann~\cite{GarckeKnopfWittmann} and it applies to the anisotropic Cahn-Hilliard equation above.

More precisely, we consider the limit $\eps\to 0$ of solutions to 
\begin{alignat}{2}
    \partial_t c_\eps +\Div (\ve_\eps  c_\eps) &= \Div (m(c_\eps)\nabla \mu_\eps)\ &\qquad&\text{in}\ \Omega\times (0,T), \label{eq:nlCH1}\\
    \mu_\eps &= \LL_\eps c_\eps + f'(c_\eps)&\qquad&\text{in}\ \Omega\times (0,T), \label{eq:nlCH2}\\
    m(c_\eps) \partial_{\no} \mu_\eps|_{\partial\Omega}  &=\mathbf{n}\cdot \ve_\eps c_\eps |_{\partial\Omega}&\qquad&\text{on}\ \partial\Omega\times (0,T), \label{eq:nlCH3}\\
    c_\eps|_{t=0}&= c_{0,\eps} &&\text{in}\ \Omega, \label{eq:nlCH4}
\end{alignat}
where $\ve_\eps\colon \Omega\times (0,T)\to \R^d$, $c_{0, \eps} \colon \Omega \to \R$ are given and converge to some $ \ve\colon 
 \Omega\times (0, T)\to \R^d $ and $c_0\colon \Omega \to \R$, respectively, in an appropriate sense , and  
\begin{align}\label{eq:defnL}
    \LL_\eps u(x) 
    &=\lim_{\delta \to 0} \int_{\Omega\setminus B_\delta(x)} (u(x)-u(y))k_\eps(x-y)dy.
\end{align}
The precise assumptions on $k_\eps\colon \Rd\setminus\{0\}\to [0,\infty)$ are specified in Section~\ref{sec:prelim} below. 
We note here that the assumptions on $ k_\eps$ require neither radial symmetry, integrability, regularity. 
The additional transport term $\Div (\ve_\eps  c_\eps)$ is motivated by a corresponding coupled Navier-Stokes/Cahn-Hilliard system, which will be studied in Section~\ref{sec:NSCH}. 

As limit we obtain the following local anisotropic Cahn-Hilliard equation
\begin{alignat}{2}
    \partial_t c +\Div (\ve  c) &= \Div (m(c)\nabla \mu)\ &\qquad&\text{in}\ \Omega\times (0,T), \label{eq:anCH1}\\
    \mu &= - \Div (A \nabla c) + f'(c)&\qquad&\text{in}\ \Omega\times (0,T), \label{eq:anCH2}\\
    \mathbf{n} \cdot (A \nabla c)|_{\partial \Omega} &= 0  &\qquad&\text{on}\ \partial\Omega\times (0,T), \label{eq:anCH3}\\     
    m(c)\partial_{\no} \mu|_{\partial\Omega}  &=\mathbf{n}\cdot \ve c|_{\partial\Omega}&\qquad&\text{on}\ \partial\Omega\times (0,T), \label{eq:anCH4}\\
    c|_{t=0}&= c_{0} &&\text{in}\ \Omega, \label{eq:anCH5}
\end{alignat}
where $A$ is determined by the limit of the second moments of $k_\eps$, cf.~\eqref{eq:ajk} below. 
We show the existence of weak solutions to the nonlocal Cahn-Hillard equations \eqref{eq:nlCH1}-\eqref{eq:nlCH4} 
for sufficiently small $ \eps >0 $ and the convergence in a suitable sense 
of weak solutions to the nonlocal Cahn-Hilliard equations \eqref{eq:nlCH1}-\eqref{eq:nlCH4} to those of the corresponding anisotropic Cahn-Hilliard equation \eqref{eq:anCH1}-\eqref{eq:anCH5}. 

This paper is organized as follows. In Section~\ref{sec:prelim}, we state the assumptions and preliminaries. In Section~\ref{sec:Existence}, we show an existence result for the nonlocal Cahn-Hilliard equations \eqref{eq:nlCH1}-\eqref{eq:nlCH4}, which in particular guarantees existence for sufficiently small $ \eps >0 $.
In Section~\ref{sec:NonlocalLimit}, we show the convergence in a suitable sense 
of weak solutions to the nonlocal Cahn-Hilliard equations \eqref{eq:nlCH1}-\eqref{eq:nlCH4} to those of the corresponding anisotropic Cahn-Hilliard equation \eqref{eq:anCH1}-\eqref{eq:anCH5}. In Section~\ref{sec:NSCH}, we consider a diffuse interface model or two-phase flows with unmatched densities with nonlocal energies.
We show convergence of weak solutions to that system to those of the system 
with local energies by applying the result of Section~\ref{sec:NonlocalLimit}. Here the main novelty is that we can treat more singular kernels than in the author's contribution \cite{AbelsTerasawaNonlocalToLocal}.

\section{Assumptions and Preliminaries}\label{sec:prelim}

We first state assumptions for the kernels $k_\eps$ corresponding to the nonlocal Cahn-Hilliard equation \eqref{eq:nlCH1}-\eqref{eq:nlCH4}. We aim at imposing conditions as weak as possible for ensuring the existence and nonlocal-to-local convergence of their solutions. The conditions are variants of those in Gounoue, Kassmann and Voigt~\cite{GounoueKassmannVoigt}. 
Throughout this contribution $k_\eps\colon \Rd\setminus\{0\}\to [0,\infty)$, $\eps\in (0,1)$, satisfies the following assumptions:
\begin{enumerate}
\item[(A1)] $k_\eps (-z)=k_\eps (z)$ for all $\eps \in (0,1)$ and $z\in \R^d\setminus \{0\}$.
\item[(A2)] For some $c_0,C_0>0$ and $\underline{A}, \overline{A}\in \R^{d\times d}$ invertible we have
\begin{equation*}
  c_0\underline{k}_{\eps}(\underline{A}z)\leq k_\eps(z)\leq C_0\overline{k}_{\eps}(\overline{A}z)\quad \text{for all }z\in\R^d, z\neq 0, 
\end{equation*}
where $\underline{k}_\eps, \overline{k}_\eps\colon \Rd\setminus\{0\} \to [0,\infty)$ are radially symmetric kernels satisfying
\begin{enumerate}
\item $\int_{\R^d} \underline{k}_\eps (z) \min (1,|z|^2) \,d z=\int_{\R^d} \overline{k}_\eps (z) \min (1,|z|^2) \,d z=1$ for all $\eps\in (0,1)$.
\item $\lim_{\eps\to 0} \int_{|z|\geq \delta} \underline{k}_\eps (z)\, dz=\lim_{\eps\to 0} \int_{|z|\geq \delta} \overline{k}_\eps (z)\, dz=0$ for all $\delta>0$.
\end{enumerate}
\item[(A3)] For every $i,j\in \{1,\ldots, d\}$ the limit
\begin{alignat}{2}\label{eq:ajk}
  \lim_{\eps\to 0} \int_{\R^d} k_\eps (z)z_iz_j \, dz &=: a_{j,k}
\end{alignat}
exists and $A= (a_{j,k})_{j,k=1}^d$.
\end{enumerate}
\begin{remark}
The assumptions above are those of Gounoue, Kassmann and Voigt~\cite{GounoueKassmannVoigt}
except for the modification of the matrix $\overline{A}$ and $\underline{A}$ in (A2). 
The following kernels $k_{\eps}$, $\eps\in (0,1)$, satisfy the assumptions (A1)-(A3).
\begin{align*}
  k_\eps(z) =
  \begin{cases}
    \eps^{-d-2}k(\tfrac{z}\eps) &\text{if }0<|z|\leq \eps,\\
    \eps^{-d}|z|^{-2}k(\tfrac{z}\eps) &\text{if }\eps <|z|\leq 1,\\
    \eps^{-d}k(\tfrac{z}\eps) &\text{if }1<|z|,
  \end{cases}
\end{align*}
where
\begin{alignat}{1}
      k(z)&=k(-z)\,, \label{k-ass-one} \\
     c_0 |z|^{-d-\alpha} &\leqslant k(z)\leqslant C_0 |z|^{-d-\alpha} \,. \label{k-ass-three}
\end{alignat}
for all $z\in\Rn$, $z\neq 0$, where $\alpha$ is the order of the operator. 

The kernels $k_{\ep}$ defined in the following satisfy the assumptions (A1)-(A3). These kernels were treated in \cite{DavoliEtAlW11} and \cite{AbelsTerasawaNonlocalToLocal}
with the additional assumptions $ k_{\ep} \in W^{1, 1}(\mathbb{R}^d)$ and
nonlocal-to-local convergence of weak solutions for 
 nonlocal Cahn-Hilliard equations with these kernels was shown:
$$
k_{\ep}(x) := \frac{\eta_{\ep}(|x|)}{|x|^{2}} \quad \text{for all }x\in\R^d\setminus\{0\},
$$ 
where $(\eta_{\ep})_{\ep >0}$ is a family of mollifiers with the following properties:
\begin{alignat}{2}
&\eta_{\ep} \colon \mathbb{R} \to [0, +\infty),~~~\eta_{\ep} \in L^{1}_{loc}(\mathbb{R}),&~~~\eta_{\ep}(r) = \eta_{\ep}(-r)~~~~\text{for all }r \in \mathbb{R},\ep>0; \nonumber\\
&\int_{0}^{+\infty} \eta_{\ep}(r)r^{d-1}\,dr = \frac{2}{C_d}~~~~\text{for all }\ep >0, && \nonumber \\
&\lim_{\ep \rightarrow 0+} \int_{\delta}^{+\infty} \eta_{\ep}(r)r^{d-1}\,dr = 0~~~\text{for all }\delta > 0,&& \nonumber 
\end{alignat} 
where $C_d:= \int_{S^{d-1}} |e_1 \cdot \sigma|^2 d \mathcal{H}^{d-1}(\sigma)$. 

Various examples of $k_\eps$ satisfying
(A1)-(A3) can be found in \cite[Section~4]{GounoueKassmannVoigt}.
If we consider the kernel $k_{\ep}(B \cdot)$, where $k_{\ep}$ is radially symmetric and satisfies (A1)-(A3) and $ B $ is an invertible matrix,
it is not radially symmetric in general and satisfies (A1)-(A3).
\end{remark}

\begin{remark}
  We note that the assumptions imply
  \begin{align*}
    &\sup_{\eps \in (0,1)}\int_{\R^d} k_\eps(z) \min (1,|z|^2)\, dz \leq C <\infty\\
    &\lim_{\eps\to 0} \int_{|z|\geq \delta}k_\eps (z) =0\qquad \text{for all }\delta>0
  \end{align*}
  for some constant $C$ depending on $C_0$ above and $\ol{A}$. In particular this implies the existence of the limit in \eqref{eq:ajk} for a suitable subsequence.
\end{remark}
Moreover, we assume that $\ve_\eps \in L^2(\Omega\times(0,T))^d$ for every $\eps\in (0,1)$ and that $\ve_{\eps}$ converges weakly to $\ve$
in $ L^2(\Omega\times (0,T))^d$ for some $\ve\in L^2(\Omega\times (0,T))^d$.

For the following, we denote
\begin{align*}
    V_\eps&:=\left\{u\in L^2(\Omega): \int_\Omega\int_\Omega (u(x)-u(y))^2 k_\eps (x-y)\,dx \, dy <\infty\right\},\\
    \|u\|_{V_\eps}&:= \left(\|u\|_{L^2(\Omega)}^2 + \int_\Omega\int_\Omega (u(x)-u(y))^2 k_\eps (x-y)\,dx \, dy\right)^{\frac12}
\end{align*}
and
\begin{alignat*}{2}
 \mathcal{B}_\eps(u, v)&:= \frac12\int_\Omega\int_\Omega (u(x)-u(y))(v(x)-v(y)) k_\eps (x-y) \,d x\, dy&\quad&\text{for }u, v \in V_\eps,\\
  \E_\eps(u)&:= \frac12 \mathcal{B}_{\eps}(u, u), &\quad&\text{for }u\in V_\eps,\\
  \mathcal{B}_0(u, v) &:=  \int_\Omega (A\nabla u(x)) \cdot \nabla v(x)\,d x,&&\text{for }u, v\in H^1(\Omega),\\
  \E_0(u) &:= \frac12\mathcal{B}_0(u, u) &\quad&\text{for }u\in H^1(\Omega),
\end{alignat*}
where $A= (a_{j,k})_{j,k=1}^d$. Here $ V_{\eps} $ is a separable Hilbert space as shown in \cite[Proposition 2.6]{GounoueKassmannVoigt}.


We note that a pair $ (c, \mu) $ is a weak solution of \eqref{eq:anCH1}-\eqref{eq:anCH5} if $c\in BC_w([0,T];L^2(\Omega))\cap L^\infty(0,T;V_\eps)\cap W^1_2(0,T;(H^1(\Omega))')$ with $f'(c)\in L^2(\Omega\times (0,T))$ (in particular $c(x,t)\in (-1,1)$ almost everywhere), $\mu\in L^2(0,T;H^1(\Omega))$, and
\begin{alignat}{1}\label{eq:nCHweak}
  \langle \partial_t c(t), \varphi\rangle - \int_\Omega\ve(t) c(t)\cdot \nabla \varphi\,dx = -\int_\Omega m(c(t)) \nabla \mu(t)\cdot \nabla \varphi\,dx 
\end{alignat}
for all $\varphi \in H^1(\Omega)$ and for almost every $t\in (0,T)$ as well as
\begin{equation*}
  \int_\Omega \mu(t)\varphi \,dx = \mathcal{B}_0 (c(t), \varphi) + \int_\Omega f'(c(t))\varphi \,dx 
\end{equation*}
for all $\varphi \in H^1(\Omega)$ and for almost every $t\in (0,T)$ and $c|_{t=0}=c_0$.

An essential point in the analysis is to generalize the following lemma from \cite[Lemma~3.3]{DavoliEtAlW11} to this anisotropic situation:
\begin{lemma}\label{lem:3.3}
  For every $\varphi, \zeta\in H^1(\Omega)$ it holds that
  \begin{align*}
    \lim_{\eps\to 0} \E_\eps(\varphi) &= \E_0(\varphi),\quad
    \lim_{\eps \to 0} \mathcal{B}_\eps (\varphi, \zeta) = \mathcal{B}_0 (\varphi, \zeta).
  \end{align*}
  Moreover, the matrix $A$ is symmetric and positive definite.
  Furthermore, for every sequence $(\varphi_n)_{n \in \mathbb{N}}\subseteq L^2(\Omega)$, $\eps_n>0$, $n\in\N$, tending to zero and $\varphi \in L^2(\Omega)$ it holds that
  \begin{alignat*}{2}
    \sup_{n\in\N} \E_{\eps_n} (\varphi_n) <+\infty & \quad \Rightarrow &\quad& (\varphi_n)_{n \in \mathbb{N}} \text{ is relatively compact in } L^2(\Omega),\\
    \sup_{n\in\N} \E_{\eps_n} (\varphi_n) <+\infty,~~\varphi_n \to_{n \rightarrow \infty} \varphi \quad \text{in } L^2(\Omega) &\quad \Rightarrow & \quad& \varphi \in H^1(\Omega) ~~\mathrm{and}~~\E_0(\varphi) \leq \liminf_{n \rightarrow \infty} \E_{\eps_n} (\varphi_n).
  \end{alignat*}
\end{lemma}
\begin{proof}
The first assertion that $\lim_{\eps\to 0} \E_\eps(\varphi) = \E_0(\varphi)$ holds for every $\varphi, \zeta\in H^1(\Omega)$ is proven in Theorem 3.4 in \cite{GounoueKassmannVoigt}. The second assertion that $ \lim_{\eps \to 0} \mathcal{B}_\eps (\varphi, \zeta) = \mathcal{B}_0 (\varphi, \zeta) $ holds for every $\varphi, \zeta\in H^1(\Omega)$ easily follows from the first using the polarization identity.

Next we prove that for all sequences $(\varphi_n)_{n\in\N} \subseteq L^2(\Omega)$, $(\eps_n)_{n\in\N}$, $0<\eps_n\to_{n\to\infty} 0$, it holds that 
$$ 
\sup_{n\in\N} \E_{\eps_n} (\varphi_n) <+\infty \quad \Rightarrow \quad (\varphi_n)_{n\in\N} \text{ is relatively compact in } L^2(\Omega).
$$
First we give the proof in the case that $\underline{A}$ is the identity matrix.
Since 
$$ 
\frac{1}{4} \int_{\Omega} \int_{\Omega} (\varphi_\ep(x) - \varphi_\ep(y))^2 k_{\ep}(x-y)\,dx\,dy \leq C 
$$ 
for some $C>0$ and for any $0<\ep<1$, we have
$$ 
\int_{\Omega} \int_{\Omega} \frac{\underline{k_{\ep}}(x-y)}{|x-y|^2}  \min(1, |x-y|^2) (\varphi_\ep(x) - \varphi_\ep(y))^2 \,dx\,dy \leq \frac{C}{c_0}
$$
for any $0<\ep<1$. 
We know from the assumptions that
$\int_{\mathbb{R}^d} \underline{k_{\ep}}(x) \min(1, |x|^2) \,dx =1$ for any $0<\ep<1$ and $ \lim_{\ep \rightarrow 0} \int_{|x| > \delta} \underline{k_{\ep}}(x) \min(1, |x|^2) = 0 $ for any $\delta>0$. Then using the first statement in \cite[Theorem~1.2]{PonceNonlocalPoincare}, we have that  
$(\varphi_n)_{n\in\N}$ is relatively compact in $L^2(\Omega)$.
When $\underline{A}$ is a general invertible matrix, we have $(\varphi_{n} (\underline{A}^{-1} \cdot))_{n\in\N}$ is relatively compact in $L^2(\underline{A} \Omega)$, where $\underline{A} \Omega $ is a domain obtained by applying the linear map $ \underline{A} $ to $ \Omega $.
Similarly as above we use change of variables in the integration and conclude that $(\varphi_{n})_{n\in\N}$ is relatively compact in $L^2(\Omega)$. 

By the same arguments using the result in \cite{BourgainBrezisMironescu} (cf. \cite[Lemma~3.3]{DavoliEtAlW11}) we can show that there exists some positive constant 
$ C>0 $ such that $ C \int_{\Omega} | \nabla \varphi(x) |^2 \, dx \leq \E_0(\varphi) $ holds for $ \varphi \in H^1(\Omega) $. This easily implies that $A$ is positive definite by taking suitable linear functions $\varphi$.

We finally prove the last assertion  
$$ \sup_{n\in\N} \E_{\eps_n} (\varphi_n) <+\infty,~~\varphi_n \to_{n \rightarrow \infty} \varphi \quad \text{in } L^2(\Omega) \quad \Rightarrow  \quad \E_0(\varphi)\leq \liminf_{n \rightarrow \infty} \E_{\eps_n} (\varphi_n).$$
This assertion is due to Foghem Gounoue, Kassmann and Voigt \cite{GounoueKassmannVoigt} and Voigt's thesis \cite{DissertationVoigt}.
See Theorem 3.4 and Lemma 3.6 in \cite{GounoueKassmannVoigt}.
The result in a more general setting is proven also in Voigt's thesis \cite{DissertationVoigt}.
We note that in this thesis it is claimed that the result also holds for nontranslation-invariant kernel, but the proof seems to be valid only in the translation invariant case.
We give here a proof of the assertion for completeness. 
The argument below closely follows that in \cite{DissertationVoigt}.\\

Let $ \varphi \in L^2(\Omega) $ and $ (\varphi_n)_{n \in \mathbb{N}} $ with
$\| \varphi_n - \varphi \|_{L^2(\Omega)} \rightarrow 0$ as $ n \rightarrow \infty$. 
Let $ f \in L^2(\Omega) $, $\alpha>0$ and $ \delta>0 $ be arbitrary.
We then define $ f^{\alpha}(x) = e_0 f * \eta_{\alpha}(x) $, where 
$ \eta_{\alpha}$ is a smooth positive, radially symmetric mollifier having support in $B_{\alpha}(0)$. Here $ e_0 f $ is the zero extension of $f$ outside of $\Omega$.
Then $ \varphi^{\alpha}, \varphi_n^{\alpha} \in C^{\infty}(\mathbb{R}^d) $ and
$ \varphi^{\alpha}_n \rightarrow_{n\to \infty} \varphi^{\alpha} $ in $C^k(\mathbb{R}^d)$ for
any $ k \in \mathbb{N}$. 
We define 
\begin{align*}
\Omega_{\delta} = \{ x \in \Omega~|~\operatorname{dist} (x, \partial \Omega)> \delta \}.
\end{align*} 
If we use Jensen's inequality, Fubini's theorem and a change of variables we obtain for all
$ n \in \mathbb{N} $ and $ \alpha \in (0, \delta) $
\begin{align*}
&\int_{\Omega_{\delta}} \int_{\Omega_{\delta}} (\varphi_n^{\alpha}(x) - \varphi_n^{\alpha}(y))^2 k_{\ep_n}(x-y)\,dy\,dx \\ 
&= \int_{\Omega_{\delta}} \int_{\Omega_{\delta}} \left( \int_{B_{\alpha}} (\varphi_n(x-z) - \varphi_n(y-z)) \eta_{\alpha}(z) \,dz \right)^2 k_{\ep_n}(x-y)\,dy\,dx \\
& \leq \int_{B_{\alpha}} \int_{\Omega_{\delta}} \int_{\Omega_{\delta}} (\varphi_n(x-z) - \varphi_n(y-z))^2 k_{\ep_n}(x-y) \,dy\,dx\,\eta_{\alpha}(z)\,dz \\
&\leq \int_{B_{\alpha}} \int_{\Omega} \int_{\Omega} (\varphi_n(x') - \varphi_n(y'))^2 k_{\ep_n}(x'-y')\,dy'\,dx'\,\eta_{\alpha}(z)\,dz \\
&= \int_{\Omega} \int_{\Omega} (\varphi_n(x') - \varphi_n(y'))^2 k_{\ep_n}(x'-y')\,dy'\,dx'= 2 \E_{\ep_n}(\varphi_n).
\end{align*}
By Taylor's formula, for each $\alpha \in (0, \delta)$  there exist some positive constants $C^{\delta}_n$ and $C^{\delta}$ depending on $\alpha$ such that
\begin{align*}
\left| \varphi^{\alpha}_n(x) - \varphi^{\alpha}_n(y) - 
\sum_{i=1}^d \partial_{x_i} \varphi^{\alpha}_n(x) (x-y)_i\right| \leq C^{\delta}_n |x-y|^2 \\
\left| \varphi^{\alpha}(x) - \varphi^{\alpha}(y) - \sum_{i=1}^{d} \partial_{x_i} \varphi^{\alpha}(x)(x-y)_i \right| \leq C^{\delta} |x-y|^2 
\end{align*}
where $C^{\delta}_n \rightarrow C^{\delta}$ as $ n \rightarrow \infty $
for $x, y \in \Omega_{\delta} $ and $|x - y| \leq \frac{\delta}2$. 
Thus we have
\begin{align*}
&\liminf_{n \rightarrow \infty} 2 \E_{\ep_n}(\varphi_n) 
 \geq \liminf_{n \rightarrow \infty} \int_{\Omega_{\delta}} \int_{\Omega_{\delta}} (\varphi^{\alpha}_n(x) - \varphi^{\alpha}_n(y))^2 k_{\ep_n}(x-y)\,dy\,dx \\
&\geq \int_{\Omega_{\delta}} \liminf_{n \rightarrow \infty} \sum_{i,j=1}^d
\partial_{x_i} \varphi^{\alpha}_n(x) \partial_{x_j} \varphi^{\alpha}_n(x) \int_{|x-y|\leq \frac{\delta}2 \cap \Omega_{\delta}} (x-y)_i (x-y)_j k_{\ep_n}(x-y)\,dy\,dx \\ 
&= \int_{\Omega_{\delta}} \sum_{i, j=1}^d a_{i, j} \partial_{x_i} \varphi^{\alpha}(x) \partial_{x_j} \varphi^{\alpha}(x)\,dx. 
\end{align*}
Using $A=(a_{i, j})_{i,j=1}^d$, the above inequality is equivalent to
\begin{align*}
\liminf_{n \rightarrow \infty} \E_{\ep_n}(\varphi_n) \geq \frac12 \int_{\Omega_{\delta}} ( A \nabla \varphi^{\alpha}(x))\cdot \nabla \varphi^{\alpha}(x)  \,dx
\end{align*} 
Since 
$ \sup_{n\in\N} \E_{\eps_n} (\varphi_n) <+\infty $, using the estimate of the kernel from below and the second statement in \cite[Theorem~1.2]{PonceNonlocalPoincare}, we have
$ \varphi(\underline{A}^{-1} \cdot) \in H^1(\underline{A} \Omega) $. Hence we have
$ \varphi \in H^1(\Omega) $. Thus $\nabla \varphi^{\alpha} \rightarrow \nabla \varphi $ in $L^2(\Omega_{\delta})$ as
$ \alpha \in (0, \delta) $ tends to zero. 

Because of $ \varphi \in H^1(\Omega) $, by absolute continuity of the Lebesgue integral, for any $ \sigma>0 $ there is some $ \delta>0 $ such that for
\begin{align*}
\frac12 \int_{\Omega \setminus \Omega_{\delta}} ( A \nabla \varphi(x))\cdot \nabla \varphi(x) \,dx < \sigma
\end{align*}
Altogether, 
\begin{align*}
 \liminf_{n \rightarrow \infty} \E_{\eps_n} (\varphi_n) \geq \E_0(\varphi) - \sigma. 
\end{align*}
Since $\sigma$ is arbitrarily small, the assertion follows.
\end{proof}


Using Theorem \ref{lem:3.3}, we can prove the following lemma completely analogously to Davoli, Ranetbauer, Scarpa, and Trussardi \cite{DavoliEtAlNonlocalLocalCHPeriodic}. Actually, it also follows by a simple comparison with the radially symmetric case (i.e., (A2)) and change of variables similarly as before. 
\begin{lemma}\label{lem:3.4}
For any $\delta>0$, there exists some $C_{\delta}>0$ and $\eps_{\delta}>0$ such that for any $\eps_1, \eps_2 \in (0, \eps_{\delta})$ and $\varphi_1,\varphi_2\in L^2(\Omega)$ we have
\begin{equation}
\| \varphi_{1} - \varphi_{2} \|^2_{L^2(\Omega)} \leq \delta (\E_{\eps_1}(\varphi_{1}) + \E_{\eps_2}(\varphi_{2}))
+ C_{\delta} \| \varphi_{1} - \varphi_{2} \|^2_{(H^1(\Omega))'}. 
\end{equation}  
\end{lemma}

\section{Existence of Solutions}\label{sec:Existence}

In the following we show existence of weak solution to \eqref{eq:nlCH1}-\eqref{eq:nlCH4} for some fixed $\mathcal{L}$ instead of $\mathcal{L}_\eps$ in a more general situation. More precisely, we consider weak solutions of
\begin{alignat}{2}
    \partial_t c +\Div (\we  c) &= \Div (m(c)\nabla \mu)\ &\qquad&\text{in}\ \Omega\times (0,T), \label{eq:nlCH1'}\\
    \mu &= \LL c + f'(c)&\qquad&\text{in}\ \Omega\times (0,T), \label{eq:nlCH2'}\\
    m(c) \partial_{\no} \mu|_{\partial\Omega}  &=\mathbf{n}\cdot \we c|_{\partial\Omega}&\qquad&\text{on}\ \partial\Omega\times (0,T), \label{eq:nlCH3'}\\
    c|_{t=0}&= c_{0} &&\text{in}\ \Omega \label{eq:nlCH4'}
\end{alignat}
in the sense that
\begin{alignat}{1}\label{eq:nCHweak}
  \langle \partial_t c(t), \varphi\rangle - \int_\Omega\we(t) c(t)\cdot \nabla \varphi\,dx = -\int_\Omega m(c(t)) \nabla \mu(t)\cdot \nabla \varphi\,dx 
\end{alignat}
for all $\varphi \in H^1(\Omega)$ and for almost every $t\in (0,T)$ as well as
\begin{equation}\label{eq:nCHweak2}
  \int_\Omega \mu(t)\varphi \,dx = \mathcal{B} (c(t), \varphi) + \int_\Omega f'(c(t))\varphi \,dx 
\end{equation}
for all $\varphi \in V$ and for almost every $t\in (0,T)$, where
\begin{align*}
  \mathcal{B} (u,v)&= \frac12\int_\Omega \int_\Omega k(x,y)(u(x)-u(y))(v(x)-v(y))\, dx\,dy\quad \text{for all }u,v\in V,\\
  V&:= \left\{u \in L^2(\Omega): \mathcal{E}(u)<\infty\right\}\quad \text{with}\\
  \|u\|_{V} &:= \left(\|u\|_{L^2(\Omega)}^2+ \mathcal{B}(u,u)\right)^{\frac12}.
\end{align*}
Here $V$ is a separable Hilbert space as shown in \cite{GounoueKassmannVoigt}.

For the following we assume that for some sufficiently small $\delta_0>0$ (which will be chosen in the proof of Theorem~\ref{thm:Existence} below) there is some $C(\delta_0)>0$ such that
\begin{equation}\label{eq:Ehrling}
  \|u\|_{L^2(\Omega)}^2 \leq \delta_0 \mathcal{E}(u) + C(\delta_0)\|u\|_{H^{-1}_{(0)}(\Omega)}^2\quad \text{for all }u\in V,
\end{equation}
where
\begin{equation*}
  \mathcal{E}(u) = \frac12\mathcal{B}(u,u) = \frac14 \int_\Omega\int_\Omega k(x-y) (u(x)-u(y))^2\,dx\, dy.
\end{equation*}

\begin{examples}
  \begin{enumerate}
  \item For $\E=\E_\eps$ and $k=k_\eps$ the assumption \eqref{eq:Ehrling} will hold true for $\eps>0$ sufficiently small because of Lemma~\ref{lem:3.4}.
  \item If $k_\eps$ is a ``singular kernel'' as in \eqref{k-ass-three}, then $\mathcal{E}(u)\geq c_0\|u\|_{H^{\alpha/2}(\Omega)}^2$ and \eqref{eq:Ehrling} follows from interpolation inequalities or Ehrling's lemma.
  \end{enumerate}
\end{examples}
\begin{theorem}\label{thm:Existence}
  Let $\Omega\subseteq \R^d$ be a bounded domain with $C^2$-boundary and $T>0$. Moreover, let $f\colon [-1,1]\to \R $ be continuous, twice continuously differentiable in $(-1,1)$ such that $-\kappa:=\inf_{s\in (-1,1)} f''(s)>-\infty$ and $m\colon [-1,1]\to (0,\infty)$ be continuous. Furthermore, let $k$ satisfy (A1)-(A2) (for fixed $\eps$). Then, if \eqref{eq:Ehrling} holds true for some $\delta_0\in (0,\frac1\kappa)$, then for any $c_{0}\in L^2(\Omega)$ with $c_0(x)\in [-1,1]$ almost everywhere, $m_\Omega:=\tfrac1{|\Omega|}\int_\Omega c_0(x)\, dx\in (-1,1)$ and
  \begin{equation*}
    E(c_0) := \frac14\int_\Omega\int_\Omega k(x,y)(c_0(x)-c_0(y))^2\,dx\, dy+\int_\Omega f(c_0(x))\, dx<\infty 
  \end{equation*}
  there is a weak solution $c\in BC_w([0,T];L^2(\Omega))\cap L^\infty(0,T;V)\cap W^1_2(0,T;(H^1(\Omega))')$ with $f'(c)\in L^2(\Omega\times(0,T))$, $\mu \in L^2(0,T;H^1(\Omega))$ of \eqref{eq:nlCH1'}-\eqref{eq:nlCH4'} satisfying the energy inequality
  \begin{align*}
    E(c(t))+\int_0^t \int_\Omega m(c)|\nabla \mu|^2\, dx\, dt\leq E(c_0) + \int_0^t \int_\Omega c\we\cdot \nabla \mu \,dx \,d\tau 
  \end{align*}
  for all $t\in [0,T]$.
\end{theorem}

For the following we use that $f(s)= f_0(s)- \kappa \frac{s^2}2$ for all $s\in [-1,1]$, where $\kappa = -\inf_{s\in (-1,1)} f''(s)$. Then $f_0\in C^0([-1,1])\cap C^2((-1,1))$ is convex. We extend $f_0$ outside of $[-1,1]$ by $+\infty$.

We prove the existence result in two steps: First we regularize $f_0$ as follows: For $\lambda\in (0,1)$ let
\begin{alignat*}{1}
g_\lambda(s)&= \begin{cases}
    f_0'(1-\lambda)+\frac1\lambda (s-1+\lambda), &\text{if } s\geq 1-\lambda\\
    f'(s) &\text{if } s\in (-1+\lambda, 1-\lambda),\\
    f_0'(-1+\lambda)+\frac1\lambda (s+1-\lambda), &\text{if } s\leq -1+\lambda.
\end{cases}\\
f_{0,\lambda}(s) &= f_0(0)+\int_0^s g_\lambda(r)\, dt
\end{alignat*}
for all $s\in\R$. Then $f'_{0,\lambda}= g_\lambda\colon \R\to \R$ is globally Lipschitz continuous,
\begin{equation}\label{eq:Coercivflambda0}
    f_{0,\lambda} (s) \geq \frac1{2\lambda} s^2 - C_\lambda \qquad\text{for all }s\in\R
    \end{equation}
    and some $C_\lambda>0$ and
\begin{equation*}
    \lim_{\lambda \to 0} f'_{0,\lambda} (s) = 
    \begin{cases} 
    +\infty &\text{if } s\geq 1,\\ 
    f_0'(s) &\text{if } s\in (-1,1),\\
    -\infty &\text{if } s\leq -1.
    \end{cases}
\end{equation*}
Finally, we set $f_\lambda(s)= f_{0,\lambda}(s)-\kappa \frac{s^2}2$ for all $s\in\R$.

It is well-known that for every $\tilde m\in (-1,1)$ there are $c_m,C_m$ such that
\begin{equation*}
|f'(s)|\leq c_{\tilde m} f'(s)(s-\tilde m) +C_{\tilde m}\qquad \text{for all }s\in(-1,1). 
\end{equation*}
Indeed this can be shown by considering the inequality on $(-1,-1+\eps]$, $[-1+\eps, 1-\eps]$, and $[1-\eps, 1)$ for some sufficiently small $\eps>0$ using $\lim_{s\to \pm 1} f'(s)= \pm \infty$ and $\tilde m \in (-1+\eps,1-\eps)$.
From this it is easy to conlcude that there are $c_{m_\Omega},C_{m_\Omega}>0$ such that for all $\lambda>0$ sufficiently small
\begin{equation*}
|f'_\lambda(s)|\leq c_{m_\Omega} f'_\lambda (s)(s-m_\Omega) +C_{m_\Omega}\qquad \text{for all }s\in\R. 
\end{equation*}

Next we use an implicit time discretization: For given 
$$c_k\in L^2_{(m_\Omega)}(\Omega):=\{u\in L^2(\Omega): \tfrac1{|\Omega|} \int_\Omega u(x)\, dx=m_\Omega\}\quad \text{with }E(c_k)<\infty,
$$ 
$k\in \{0,\ldots, N-1\}$ and $h=\frac{T}N$, $N\in\N$, let $c_{k+1}\in V\cap L^2_{(m_\Omega)}(\Omega)$ be a minimizer of
\begin{equation*}
  \mathcal{F}_h(c):= \frac{h}2\int_\Omega m(c_k)|\nabla \mu_0|^2\, dx + E_{0,\lambda}(c)- \int_{\Omega}\kappa c_k c\,dx ,
\end{equation*}
where $\mu_0\in H^1_{(0)}(\Omega):=H^1(\Omega)\cap L^2_{(0)}(\Omega)$ is the weak solution of
\begin{alignat*}{2}
  \Div (m(c_k)\nabla \mu_0) &= \frac1h (c-c_k) + \Div (\we_k c_k)&\quad &\text{in }\Omega,\\
  \no\cdot m(c_k)\nabla \mu_0|_{\partial\Omega} &= \no \cdot \we_k c_k &&\text{on }\partial\Omega,
\end{alignat*}
i.e.,
\begin{equation}\label{eq:WeakNeumann1}
  \int_\Omega m(c_k)\nabla \mu_0\cdot \nabla \psi \,dx = -\frac1h\int_{\Omega} (c-c_k)\psi\,dx +\int_{\Omega} \we_k c_k\cdot \nabla \psi\,dx \quad 
\end{equation}
for all $\psi \in H^1_{(0)}(\Omega)$ 
and
\begin{equation*}
  E_{0,\lambda}(c)= \frac14\int_\Omega\int_\Omega k(x,y)(c(x)-c(y))^2\,dx\, dy+\int_\Omega f_{0,\lambda}(c(x))\, dx. 
\end{equation*}
Here $\we_k= \frac1h \int_{kh}^{(k+1)h} \we(x,\tau)\,dx$ for $k=0,\ldots, N-1$.

Because of \eqref{eq:Coercivflambda0}, it is easy to observe that $\mathcal{F}_h$ is coercive on $V\cap L^2_{(m_\Omega)}(\Omega)$. Moreover, $\mathcal{F}_h$ is sequentially weakly lower semi-continuous since $\mathcal{F}_h\colon V\cap L^2_{(m_\Omega)}(\Omega)\to \R$ is continuous and convex. This yields existence of a minimizer $c_{k+1}\in V\cap L^2_{(m_\Omega)}(\Omega)$ by the direct method of the calculus of variations.

Let $c_{k+1}\in V\cap L^2_{(m_\Omega)}(\Omega)$ be a minimizer of $\mathcal{F}_h$. Then for any $\varphi \in L^2_{(0)}(\Omega)$ with $\mathcal{E}(\varphi)<\infty$ we have 
\begin{align*}
    0 =& \frac{d}{d\eps} \mathcal{F}_h(c_{k+1}+\eps\varphi)|_{\eps=0}=h\int_\Omega m(c_k)\nabla \mu_{0,k+1}\cdot\nabla \mu_0'\,dx \\
    &+\frac12 \int_\Omega\int_\Omega k(x,y)(c_{k+1}(x)-c_{k+1}(x))(\varphi(x)-\varphi(y))\, dx\, dy+\int_\Omega (f'_{0,\lambda}(c_{k+1})-\kappa c_k)\varphi\,dx,
\end{align*}
where $\mu_{0,k+1}= \mu_0$ solves \eqref{eq:WeakNeumann1} with $c=c_{k+1}$ and $\mu_0'$ solves the derivative of \eqref{eq:WeakNeumann1} with respect to $c$ in direction of $\varphi$, i.e.,
\begin{equation*}
    \int_\Omega m(c_0)\nabla \mu_0'\cdot \nabla \psi \,dx = -\frac1h \int_\Omega \varphi \psi\, dx\qquad \text{for all }\psi \in H^1_{(0)}(\Omega).
\end{equation*}
Choosing $\psi=\mu_{0,k+1}$ we obtain
\begin{align}\nonumber
  \int_\Omega  \mu_{0,k+1} \varphi\,dx =& \frac12 \int_\Omega\int_\Omega k(x,y) (c_{k+1}(x)-c_{k+1}(y))(\varphi(x)-\varphi(y))\,dx \, dy\\\label{eq:chemical}
  &+ \int_\Omega (f_{0,\lambda}'(c_{k+1})-\kappa c_k) \varphi \,dx
\end{align}
for all $\varphi\in L^2_{(0)}(\Omega)$ with $\mathcal{E}(\varphi)<\infty$. This means that a minimizer $c_{k+1}$ satisfies the Euler-Lagrange equation
\begin{alignat*}{2}
  \frac{c_{k+1}-c_k}h+ \Div (\we_k c_k) &= \Div (m(c_k) \nabla \mu_{0,k+1})&\quad & \text{in }\Omega,\\
  \mu_{0,k+1} &=   \mathcal{L} c_{k+1} + P_0(f_{0,\lambda}'(c_{k+1})-\kappa c_k)&\quad & \text{in }\Omega,\\
  \no\cdot m(c_k)\nabla \mu_{0,k+1}|_{\partial\Omega} &=\no\cdot \we_k c_k&&\text{on }\partial\Omega
\end{alignat*}
in a weak sense.  Now we define
\begin{equation*}
  \mu_{k+1} = \mu_{0,k+1} + \frac1{|\Omega|}\int_\Omega (f_{0,\lambda}'(c_{k+1})-\kappa c_k)\,dx.
\end{equation*}
This yields
\begin{equation*}
  \int_\Omega  \mu_{k+1} \varphi\,dx = \mathcal{B}(c_{k+1},\varphi) + \int_\Omega (f_{0,\lambda}'(c_{k+1})-\kappa c_k) \varphi \,dx
\end{equation*}
for all $\varphi\in L^2(\Omega)$ with $\mathcal{E}(\varphi)<\infty$. 
Choosing $\psi = h\mu_{0,k+1}$ in \eqref{eq:WeakNeumann1} for $\mu_0=\mu_{0,k+1}$  and $\varphi = \frac{c_{k+1}-c_k}h$ in \eqref{eq:chemical} and using that $\mu_{k+1}$ and $\mu_{0,k+1}$ differ only by a constant, we obtain
\begin{align*}
 \mathcal{B}(c_{k+1},c_{k+1}-c_k) &+ \int_\Omega f'_{0,\lambda}(c_{k+1})(c_{k+1}-c_k)\,dx - \int_\Omega \kappa c_k(c_{k+1}-c_k) + h \int_\Omega m(c_k)|\nabla \mu_{k+1}|^2\,dx \\
&= h\int_\Omega c_k\we_k\cdot \nabla \mu_{k+1}\,dx.
\end{align*}
Since $f_{0,\lambda}$ is convex, we have
\begin{equation*}
  \int_{\Omega} f_{0,\lambda}(c_{k+1}) \,dx - \int_{\Omega} f_{0,\lambda}(c_k) \,dx\leq  \int_\Omega f'_{0,\lambda}(c_{k+1})(c_{k+1}-c_k)\,dx.
\end{equation*}
Moreover, 
\begin{align*}
  \mathcal{B}(c_{k+1},c_{k+1}-c_k) &= \mathcal{E}(c_{k+1})- \mathcal{E}(c_k) + \mathcal{B}(c_{k+1}-c_k,c_{k+1}-c_k),\\
  -\int_\Omega \kappa c_k(c_{k+1}-c_k) &= -\frac{\kappa}2 \int_\Omega c_{k+1}^2\,dx+\frac{\kappa}2 \int_\Omega c_k^2\,dx +\frac{\kappa}2 \int_\Omega (c_{k+1}-c_k)^2\,dx.
\end{align*}
Altogether this yields
\begin{align}\label{eq:DiscreteEnergyInequality}
  E_\lambda (c_{k+1}) + h \int_\Omega m(c_k)|\nabla \mu_{k+1}|^2\,dx \leq E_\lambda (c_k)+ h\int_\Omega c_k\we_k\cdot \nabla \mu_{k+1}\,dx, 
\end{align}
where
\begin{equation*}
  E_\lambda(c)= \frac14\int_\Omega\int_\Omega k(x,y)(c(x)-c(y))^2\,dx\, dy+\int_\Omega f_{\lambda}(c(x))\, dx. 
\end{equation*}
Now let $c_h\colon [-h,T]\times \Omega \to\R$, $\mu_h\colon [-h,T]\times \Omega \to\R$ be piecewise-constant interpolants of $(c_k,\mu_k)$ at $t=kh$, i.e., $c_h(t)= c_k$, $\mu_h(t)=\mu_k$ for $t\in [(k-1)h,kh) $ and $k=0,\ldots, N$. Then
\begin{alignat*}{2}
  \partial_t^hc_h+ \Div (\we_h^-c_h^-)&= \Div (m(c_h^-) \nabla \mu_h)&\quad & \text{in }\Omega\times (0,T),\\
  \mu_h &=  \mathcal{L} c_h + f_{0,\lambda}'(c_h)-\kappa c_h^-&\quad & \text{in }\Omega\times (0,T),\\
  \no\cdot m(c_h^-)\nabla \mu_{k+1}|_{\partial\Omega} &=\no \cdot \we_h^-c_h^-&&\text{on }\partial\Omega\times (0,T),
\end{alignat*}
where $c_h^-(x,t)=c_h(x,t-h)$, $\partial_t^h c_h(x,t)= \frac1h (c_h(x,t+h)-c_h(x,t))$, and the second line is understood analogously as before. Using \eqref{eq:DiscreteEnergyInequality} we obtain the energy estimate
\begin{align}\label{eq:DiscreteEnergyInequality2}
  E_\lambda(c_h(t))+\int_0^t \int_\Omega m(c_h)|\nabla \mu_h|^2\, dx\, dt\leq E_\lambda(c_0) + \int_0^t \int_\Omega c_h\we\cdot \nabla \mu_h \,dx \,d\tau 
\end{align}
for all $t\in [0,T]\cap h\N_0$.
From this (together with Young's inequality) we obtain for a suitable subsequence $h\to 0$ (not relabeled)
\begin{alignat*}{2}
  c_h&\rightharpoonup_{h\to 0}^\ast c_\lambda &\qquad & \text{in }L^\infty(0,T; V),\\
  \mu_h&\rightharpoonup_{h\to 0} \mu_\lambda &\qquad & \text{in }L^2(0,T; H^1(\Omega))
\end{alignat*}
since $L^1(0, T; V)$ is separable. Moreover, by a time discrete version of the Lemma of Aubin-Lions due to Dreher and J\"ungel~\cite[Theorem 1]{DreherJuengelAubinLions} we obtain:
\begin{equation*}
  c_h-m_\Omega\to_{h\to 0} c_\lambda-m_\Omega \qquad \text{in }L^2(0,T;H^{-1}_{(0)}(\Omega)).
\end{equation*}
Next we show that $(c_h)_{h>0}$ is a Cauchy sequence in $L^2(\Omega\times (0,T))$. To this end let $h,h'>0$. Then
\begin{align*}
    \int_\Omega (\mu_h-\mu_{h'})\varphi\,dx = \mathcal{B}(c_h-c_{h'}, \varphi)+ \int_\Omega (f'_{0,\lambda}(c_h)-f'_{0,\lambda}(c_{h'}))\varphi\,dx -\kappa \int_\Omega (c_h-c_{h'})\varphi\,dx 
\end{align*}
for all $\varphi\in L^2_{(0)}(\Omega)$ with $\mathcal{E}(\varphi)<\infty$. Choosing $\varphi= c_h-c_{h'}$ we obtain
\begin{align*}
  &\int_{Q_T}(\mu_h-\mu_{h'})(c_h-c_{h'})\,d(x,t) + \frac{C(\delta_0)}{\delta_0}\|c_h-c_{h'}\|_{L^2(0,T;H^{-1}_{(0)})}^2\\
  &\quad \geq \mathcal{E}(c_h-c_{h'}) + \frac{C(\delta_0)}{\delta_0}\|c_h-c_{h'}\|_{L^2(0,T;H^{-1}_{(0)})}^2 -\kappa \|c_h-c_{h'}\|_{L^2(Q_T)}^2
    \geq (\tfrac1{\delta_0} -\kappa) \|c_h-c_{h'}\|_{L^2(Q_T)}^2
\end{align*}
because of \eqref{eq:Ehrling}.
Hence, if $\delta_0 <\frac1\kappa$, $(c_h)_{h>0}$ is a Cauchy sequence in $L^2(\Omega\times (0,T))$ since the left-hand side convergences to zero as $h,h'\to 0$. Now, using the Lipschitz continuity of $f_{0,\lambda}$, it is easy to pass to the limit $h\to 0$ in the equations for $(c_h,\mu_h)$ and show that $(c_\lambda,\mu_\lambda)$ solve \eqref{eq:nCHweak}-\eqref{eq:nCHweak2} with $f_\lambda(s) =f_{0,\lambda}(s)-\tfrac{\kappa}2s^2$ instead of $f(s)$. More precisely,
\begin{alignat*}{1}
  \langle \partial_t c_\lambda(t), \varphi\rangle - \int_\Omega\we(t) c_\lambda(t)\cdot \nabla \varphi\,dx = -\int_\Omega m(c_\lambda(t)) \nabla \mu_\lambda(t)\cdot \nabla \varphi\,dx 
\end{alignat*}
for all $\varphi \in H^1(\Omega)$ and for almost every $t\in (0,T)$ as well as
\begin{equation}\label{eq:mlambda}
  \int_\Omega \mu_\lambda(t)\varphi \,dx = \mathcal{E} (c_\lambda(t), \varphi) + \int_\Omega f'_\lambda (c(t))\varphi \,dx 
\end{equation}
for all $\varphi \in V$ and for almost every $t\in (0,T)$. Furthermore, passing to the limit in \eqref{eq:DiscreteEnergyInequality2}, using the weak lower semi-continuity of $\mathcal{E}$ in $V$ and the strong convergence in $L^2(\Omega)$ as well as $\sqrt{m(c_\lambda)}\nabla \mu_\lambda\rightharpoonup_{\lambda\to 0} \sqrt{m(c)}\nabla \mu$, we obtain
\begin{align}\label{eq:ApproxEnergyInequality}
  E_\lambda(c_\lambda(t))+\int_0^t \int_\Omega m(c_\lambda)|\nabla \mu_\lambda|^2\, dx\, dt\leq E_\lambda(c_0) + \int_0^t \int_\Omega c_\lambda \we\cdot \nabla \mu_\lambda \,dx \,d\tau 
\end{align}
for all $t\in [0,T]$.

In order to pass to the limit $\lambda\to 0$, we obtain first of all from \eqref{eq:ApproxEnergyInequality} together with Young's inequality that
\begin{align}\label{eq:Bddcmulambda}
    (c_\lambda)_{\lambda\in (0,1)} \subset L^\infty(0,T;L^2(\Omega)),\quad (\nabla \mu_\lambda)_{\lambda\in (0,1)} \subset L^2(0,T;L^2(\Omega))
\end{align}
are bounded. 
Moreover, we need two additional estimates for $\|f_{0,\lambda}(\tc_\lambda)\|_{L^2(\Omega\times (0,T))}$ and \\
$\|\mu_\lambda\|_{L^2(\Omega\times (0,T))}$. 
Using \eqref{eq:Coercivflambda0} yields
\begin{align}
& \int_{\Omega} f'_{0,\lambda}(c_\lambda(x, t))(c_\lambda(x, t) - m_\Omega) dx \geq C_1 \| f'_{0,\lambda}(c_\lambda) \|_{L^1(\Omega)} - C_2.
\end{align}
Now using $\varphi = c_\lambda(t)-m_\Omega$ as a test function in \eqref{eq:mlambda}, we obtain:
\begin{alignat*}{1}
    &\int_\Omega \mu_{0,\lambda}(t) (c_{\lambda}(t)-m_\Omega)\,dx +\kappa \int_\Omega  (c_{\lambda}(t)-m_\Omega)^2 \,dx \\
    &\quad = \mathcal{B} (c_\lambda(t), c_\lambda (t)) + \int_\Omega f'_{0,\lambda} (c(t))(c_\lambda(t)-m_\Omega) \,dx\\
    &\quad \geq C_1 \| f'_{0,\lambda}(c_\lambda) \|_{L^1(\Omega)} - C_2,
\end{alignat*}
where $\mu_{0,\lambda}= \mu_\lambda - \tfrac1{|\Omega|} \int_\Omega \mu_\lambda \, dx$. Using the Poincar\'e's inequality and the bounds in \eqref{eq:Bddcmulambda} one obtains that the left-hand side is bounded in $L^2(0,T)$. Hence we obtain that $(f'_{0,\lambda}(c_\lambda))_{\lambda\in (0,1)}$ is bounded in $L^2(0,T;L^1(\Omega))$. Now choosing $1$ as test function in \eqref{eq:mlambda} yields
\begin{align*}
    \int_\Omega \mu_\lambda(t) \,dx = \int_\Omega (f'_{0,\lambda} (c_\lambda)- \kappa c_\lambda)\,dx\quad \text{for almost every }t\in (0,T).
\end{align*}
Hence $\int_\Omega \mu_\lambda(t) \,dx$, $\lambda\in (0,1)$, is bounded in $L^2(0,T)$. Hence another application of Poincar\'e's inequality yields the boundedness of
\begin{equation*}
    (\mu_\lambda)_{\lambda\in (0,1)}\subset L^2(0,T;H^1(\Omega)).
\end{equation*}
Next we test \eqref{eq:mlambda} with $f'_{0,\lambda}(c_\lambda(t))$ and obtain
 \begin{alignat*}{1}
     \mathcal{B}(c_\lambda(t), f'_{0,\lambda}(c_\lambda(t)))+ \int_\Omega |f_{0,\lambda}'(c_\lambda(t))|^2\,dx \leq \|\mu_\lambda(t)+\kappa c_\lambda(t)\|_{L^2(\Omega)}\|f'_{0,\lambda}(c_\lambda(t))\|_{L^2(\Omega)},
 \end{alignat*}
 where 
 \begin{alignat*}{1}
    & \mathcal{B}(c_\lambda(T), f'_{0,\lambda}(c_\lambda(t)))\\
    &= \frac14 \int_\Omega\int_\Omega k(x-y)(c_\lambda(x,t)-c_\lambda(y,t))(f'_{0,\lambda}(c_\lambda(x,t))-f'_{0,\lambda}(c_\lambda(y,t)))\,dx \, dy\geq 0
 \end{alignat*}
 since $f'_{0,\lambda}\colon \R\to \R$ is non-decreasing. Thus we obtain the uniform boundedness of 
 \begin{equation*}
    (f'_{0,\lambda}(c_\lambda))_{\lambda\in (0,1)}\subset L^2(0,T;L^2(\Omega)).
\end{equation*}
As before this yields for a suitable subsequence
\begin{alignat*}{2}
  c_\lambda&\rightharpoonup_{\lambda\to 0}^\ast c &\qquad & \text{in }L^\infty(0,T; V),\\
  \mu_\lambda &\rightharpoonup_{\lambda\to 0} \mu &\qquad & \text{in }L^2(0,T; H^1(\Omega)),\\
  c_\lambda&\rightharpoonup_{\lambda\to 0}^\ast c &\qquad & \text{in }L^2(0,T; L^2(\Omega)).
\end{alignat*}
Moreover, 
\begin{equation*}
    f'_{0,\lambda}(c_\lambda)\to_{\lambda\to 0} \xi\qquad  \text{in } L^2(0,T;L^2(\Omega))
\end{equation*}
for some $\xi\in L^2(0,T;L^2(\Omega))$. Using that $c_\lambda(x,t)\to_{\lambda\to} c(x,t)$ almost everywhere for a suitable subsequence, one can conclude from this in a nowadays standard manner that $\xi(x,t) = f'_0(c(x,t))$ for almost every $x\in \Omega, t\in (0,T)$. E.g.\ one obtains 
\begin{align*}
    &\int_0^T\int_\Omega (\xi(x,t)- f'_0(s))(c(x,t)-s)\varphi(x,t)\,dx \,dt \\
    &= \lim_{\lambda\to 0}\int_0^T\int_\Omega (f'_{0,\lambda}(c_\lambda(x,t))- f'_0(s))(c_\lambda(x,t)-s)\varphi(x,t)\,dx \,dt\geq 0
\end{align*}
for every $\varphi \in C_0^\infty(\Omega\times (0,T))$ with $\varphi\geq 0$ and every $s\in (-1,1)$. From this one concludes
\begin{align*}
    (\xi(x,t)- f'_0(s))(c(x,t)-s)\geq 0\quad \text{for almost every } x\in\Omega,t\in(0,T) \text{ and every } s\in (-1,1).
\end{align*}
This yields $\xi= f'_0(c)$ and $c(x,t)\in (-1,1)$ almost everywhere by choosing $s=-c(x,t)+\eps$ if $c(x,t)<1$ and $s=c(x,t)-\eps$ if $c(x,t)>-1$ and sending $\eps>0$ to zero.

Using this it is easy to pass to the limit in the equations and show that $(c,\mu)$ is a weak solution of \eqref{eq:nlCH1'}-\eqref{eq:nlCH4'} satisfying the stated energy inequality.

\section{Nonlocal to Local Limit}\label{sec:NonlocalLimit}

From results in Section \ref{sec:Existence}, for sufficiently small $\eps$, there exists a weak solution of the nonlocal Cahn-Hilliard equation \eqref{eq:nlCH1}-\eqref{eq:nlCH4} for any initial data $c_{0,\eps} \in L^2(\Omega)$ with $|c_{0,\eps}|\leq 1$ and $\frac1{|\Omega|}\int_\Omega c_{0,\eps}(x)\,dx\in (-1,1)$. We show here that some subsequence of weak solutions of the nonlocal Cahn-Hilliard equation \eqref{eq:nlCH1}-\eqref{eq:nlCH4} converges in an appropriate way to a weak solution of the local Cahn-Hilliard equation \eqref{eq:anCH1}-\eqref{eq:anCH5}: 
\begin{theorem}\label{maintheorem}
Let $\eps_{*}>0$. For any $\eps \in (0, \eps_{*})$ let $c_{0, \eps} \in L^{\infty}(\Omega)$ with
$ |c_{0, \eps}| \leq 1 $ almost everywhere, and
$\tfrac{1}{| \Omega |} \int_{\Omega} c_{0, \eps}(x) \,dx = m_{\Omega} $ for all $\eps \in (0, \eps_{*})$ and some $m_{\Omega} \in (-1, 1)$. Moreover, we assume that there is some $c_0 \in H^1(\Omega) $ such that $ c_{0, \eps} \rightarrow_{\eps \rightarrow 0} c $ in $L^2(\Omega)$ and $ \mathcal{E}_{\eps}(c_{0, \eps}) \rightarrow_{\eps\to 0} \mathcal{E}_0(c_0) $ and there are $\ve_\eps \in L^2(\Omega\times(0,T))^d$ for every $\eps\in (0,\eps_{*})$ and some $\ve\in L^2(\Omega\times (0,T))^d$ such that $\ve_{\eps}$ converges weakly to $\ve$
in $ L^2(\Omega\times (0,T))^d. $
If $(c_{\eps},\mu_{\eps})$ are weak solutions of the nonlocal Cahn-Hilliard equation \eqref{eq:nlCH1}-\eqref{eq:nlCH4}
with initial values $c_{0, \eps}$, then
\begin{align}
&\mu_{\eps} \rightharpoonup \mu \quad \text{in }L^2(0, T; V), \label{muweakconv}
\\
&c_{\eps} \rightarrow c \quad\text{in }C([0, T]; L^2(\Omega))~ \text{and  a.e.} \label{cstrongconv}
\end{align}
for a suitable sequnece $\eps = \eps_k \rightarrow_{k \rightarrow \infty} 0$, where $(c, \mu)$ is a weak solution of the anisotropic Cahn-Hilliard equation
\eqref{eq:anCH1}-\eqref{eq:anCH5} with initial data $c_0$, which satisfies the energy inequality 
\begin{align*}
   E_0(c(t))+\int_0^t \int_\Omega m(c)|\nabla \mu|^2\, dx\, dt\leq E_0(c_0) + \int_0^t \int_\Omega c\we\cdot \nabla \mu \,dx \,d\tau 
 \end{align*}
for all $t\in [0,T]$, where
\begin{equation*}
    E_0(c)= \frac12 \int_\Omega \nabla c\cdot A \nabla c\,dx + \int_\Omega f(c)\,dx.
\end{equation*}
\end{theorem}

\begin{proof}
One can show that $ (\mu_{\eps})_{\eps \in (0, \eps_*)} $ is bounded in $L^2(0, T; H^1(\Omega))$ adapting the arguments in Section 4.1 in \cite{DavoliEtAlW11} (See also \cite{AbelsTerasawaNonlocalToLocal}). 
We partly review the arguments here for the convenience of the reader
with changes required for the nonlocal energies involved.

Since we know from the energy inequality that $(\nabla \mu_{\eps})_{\eps \in (0, \eps_*)}\subseteq L^2 (\Omega\times (0,T))^d$ is bounded, we know from the Poincar\'e-Wirtinger inequality that it is enough to show that $ (\mu_{\eps})_{\Omega} = \frac{1}{|\Omega|} \int_{\Omega} \mu_{\eps}\,dx $ is bounded in $L^2(0, T)$.  

We define 
\begin{align*}
\mathcal{N}(c_{\eps}(t)): (H^1_{(0)}(\Omega))' \rightarrow H^1_{(0)}(\Omega) = \left\{ u \in H^1(\Omega) : \int_{\Omega} u \,dx = 0 \right\}: f \mapsto u,
\end{align*}
where $ u \in H^1_{(0)}(\Omega)$ is the solution of 
\begin{align*}
\int_{\Omega} m(c_{\eps}(t))\nabla u \cdot \nabla \psi \,dx = \left<f, \psi \right> \quad \text{for all}~~\psi \in H^1_{(0)}(\Omega).
\end{align*}
Since $m$ is strictly bounded below, there is some constant $C$ independent of $ c_{\eps}(t)$ such that 
\begin{align}
  \|\mathcal{N}(c_{\eps}(t))f\|_{H^1(\Omega)}\leq C\|f\|_{(H^1_{(0)}(\Omega))'}\qquad \text{for all }f \in (H^1_{(0)}(\Omega))'. \label{N-estimate}
\end{align}
Then testing \eqref{eq:nlCH1} $\mathcal{N}(c_\ep(t))(c_{\ep}(t) - m_{\Omega})$ (in the weak sense), \eqref{eq:nlCH2} with $c_{\ep}(t)-m_\Omega$ and taking the sum yields
\begin{align}
& \duality{\partial_t c_{\ep}(t)}{\mathcal{N}(c_\ep(t))(c_{\ep}(t) - m_\Omega)}_{H^1(\Omega)} + \mathcal{E}_{\ep}(c_{\ep}(t)) \nonumber
\\ &+ \int_{\Omega} f'_0(c_{\ep}(x, t))(c_{\ep}(x, t) - m_{\Omega})\, dx + \int_{\Omega} \theta_c c_{\ep}(x, t) (c_{\ep}(x, t) - m_{\Omega})\, dx
\nonumber \\ &= \int_{\Omega} c_{\ep}(x, t) \mathbf{v}_{\ep}(x, t) \cdot \nabla \mathcal{N}(c_\ep (t))(c_{\ep}(x, t) - m_{\Omega})\, dx, \label{testing eq:1}  
\end{align}
where 
$$
f_0(s):= f(s)+\theta_c \frac{s^2}2\qquad \text{for }s\in [-1,1]
$$
is the ``convex part'' of $f$.
Since $(\mathbf{v}_{\ep})_{\eps\in (0,\eps*)}$ converges weakly to $ \mathbf{v} $ in $ L^2(0, T; L^2(\Omega))^d $, $ \mathbf{v}_{\ep} $ is bounded in $ L^2(0, T; L^2(\Omega))^d $ for sufficiently small $\eps>0$. 
Then using \eqref{testing eq:1} to estimate $ (\mu_{\eps})_{\Omega} $, we show that $\mu_{\ep} $ is bounded in $ L^2(0, T; H^1(\Omega)) $ for sufficiently small $\eps>0$ and we can choose a subsequence such that \eqref{muweakconv} holds.
For more details of the argument, we refer to \cite{AbelsTerasawaNonlocalToLocal}.

We show that $c_{\eps} \rightarrow_{\eps \to 0} c$ strongly in $C([0, T]; L^2(\Omega))$ and a.e. for a subsequence. $\partial_t c_{\ep}$
is bounded in $L^2(0, T;(H^1(\Omega))')$ using the weak form of \eqref{eq:nlCH1} and the energy bound. Furthermore $(c_{\eps})_{\eps \in (0, \eps_*)}$ is bounded in $ L^{\infty}(0, T; L^2(\Omega)) $ since $ |c_{\eps}(x, t)|<1 $ a.e. in $Q_T$. Since $L^2(\Omega)$ is compactly embedded in $(H^1(\Omega))'$, we have $c_{\eps} \rightarrow_{\eps\to 0} c$ in $C([0, T]; (H^1(\Omega))')$ for a suitable subsequence by the Aubin-Lions lemma. Using Lemma \ref{lem:3.4} and the bounds on the energies from the energy inequality of the nonlocal Cahn-Hilliard equation, we have $ c_{\eps} \rightarrow_{\eps\to 0} c $ in $C([0, T]; L^2(\Omega))$ and a.e. for a suitable sequence. 

Since $\partial_t c_{\eps}$ is bounded in $L^2(0, T; (H^1(\Omega))')$, $\partial_t c_{\eps} $ converges weakly to $\partial_t c$ in $L^2(0, T; (H^1(\Omega))')$. Because of $\ve_{\eps} \rightarrow \ve$ weakly in $L^2(Q_T)^d$,
 $ c_{\eps} \rightarrow c $ almost everywhere, and  $|c_\eps(x,t)|\leq 1$, we have that $ \ve_{\eps} c_{\eps} $ converges weakly to $ \ve c $ in $L^2(Q_T)^d$ as $\eps\to 0$. We also have that $m(c_{\eps}) \nabla \mu_{\eps}$ converges weakly to $m(c)\nabla \mu$ in $L^2(Q)^d$. We use these facts to pass to the limit in the weak form of the equation
\begin{align} \label{eq:order}
\partial_t c_{\eps} + \Div (\ve_{\eps} c_{\eps}) = \Div(m(c_{\eps}) \nabla \mu_{\eps})~~~~\mathrm{in}~~\Omega \times (0, T).
\end{align}
By the argument in \cite[Chapter 5]{DavoliEtAlW11} (see also \cite{AbelsTerasawaNonlocalToLocal}),
we can show that there exists 
$ \xi,\eta \in L^2(0, T; L^2(\Omega)) $ such that
\begin{align}
f_0'(c_{\eps}) &\rightharpoonup \xi~~\mathrm{in}~~L^2(0, T; L^2(\Omega)) \\
\int_0^T \mathcal{B}_{\eps} (c_{\eps},\psi(t))\,dt &= \int_{\Omega \times (0,T)} \eta \psi \,d (x,t)\quad \text{for all }\psi \in L^2(0, T; H^1(\Omega))
\end{align}
for a suitable subsequence. 
Using $ c_{\eps} \rightarrow c $ in $C([0, T]; L^2(\Omega))$ one can deduce $f'_0(c)\in L^2(\Omega\times (0,T))$ and 
$ f_0'(c_{\eps}) \rightarrow f_0'(c) $ almost everywhere by the same arguments as e.g.\ in \cite[page 1093]{AbelsBosiaGrasselli}. Therefore in $L^q(Q_T)$ for every $1\leq q<2$. Hence $\xi = f_0'(c)
$. Passing to the limit in the weak formulation of \eqref{eq:order} we have 
\begin{align}
\langle \partial_t c(t), \psi) \rangle_{H^1(\Omega)} +
\int_{\Omega} m(c(x, t))\nabla \mu(x, t)\cdot \nabla \psi(x)\,dx = \int_{\Omega} c(x, t) v(x, t) \cdot \nabla \psi(x)\,dx 
\end{align}
for every $\psi \in H^1(\Omega)$ and for almost every $ t \in (0, T) $ and that $ \mu = \eta + f_0'(c) - \kappa \tc $.

It only remains to show that
\begin{align}
c \in L^{\infty}(0, T; H^1(\Omega)) \cap L^2(0, T; H^2(\Omega))
\end{align}
and $\eta = -\Div(A \nabla c)$.
Because of $c_{\eps} \rightarrow_{\eps\to 0} c$ in $C([0, T]; L^2(\Omega))$,
Lemma \ref{lem:3.3} and the energy estimate, we have $c(t)\in H^1(\Omega)$ for every $t\in [0,T]$ and
\begin{align}
\| \mathcal{E}_0(c) \|_{L^{\infty}(0, T)} 
\leq \liminf_{\eps \rightarrow 0} \| \mathcal{E}_0^{\eps} (c_{\eps}) \|_{L^{\infty}(0, T)} \leq M
\end{align}
Hence we have $ c \in L^{\infty}(0, T; H^1(\Omega)) $.
Since $\mathcal{E}_\ep (c_\ep)$ is quadratic in $c_\ep$, we have
\begin{align*}
\int_0^T \mathcal{E}_\ep (c_\ep(t)) \,dt + \int_0^T \mathcal{B}_\ep (c_\ep (t), \psi(t) - c_\ep(t))\,dt \leq \int_0^T \mathcal{E}_\ep(\psi(t))\,dt
\end{align*}
for any $\psi \in L^2(0,T;H^1(\Omega))$ with $(\psi(t))_{\Omega} =m_\Omega$ for almost all $t\in (0,T)$. 
Since $c_\ep \rightarrow_{\eps\to 0} c$ in $C([0, T]; L^2(\Omega))$, using Lemma \ref{lem:3.3} and Fatou's lemma, we derive
\begin{align*}
\int_{Q_T} \mathcal{E}_0(c(t))\,dt + \int_{Q_T} \eta(x,t) (\psi - c)(t, x) d(x, t) 
\leq \int_{Q_T} \mathcal{E}_0(\psi(t))\,dt
\end{align*}
for every $\psi \in L^2(0, T; H^1(\Omega))$ with $(\psi(t))_{\Omega} = m_\Omega$ for almost every
$ t \in (0, T)$. If we take $\psi(x,t)=c(x,t)+h \chi(t) \tau(x)$, where $ h \in \mathbb{R}$,
$ \chi \in C([0, T])$ and $ \tau \in H^1_{(0)}(\Omega) $, divide by $h$, and pass to the limit $ h \rightarrow 0$,
we obtain 
\begin{align*}
\int_{\Omega} \eta(x, t) \tau(x)\,dx = \int_{\Omega} A \nabla c(x, t) \cdot \nabla \tau(x)\,dx
\end{align*}
for a.e. $t \in (0, T)$ and for all $\tau \in H^1_{(0)}(\Omega)$.
Finally, by classical elliptic regularity theory, we conclude that $c \in L^2(0, T; H^2(\Omega))$,
$ \eta = - \Div(A \nabla c)$ and $ \left. \no \cdot (A \nabla c) \right|_{\partial \Omega} = 0$.
\end{proof}

\section{Application to a Navier-Stokes/Cahn-Hilliard System}\label{sec:NSCH}

In this section, we consider the application of our result to a Navier-Stokes/Cahn-Hilliard system. 

We consider the following equation, which is a nonlocal variant of a diffuse interface model 
of two-phase flows with unmatched densities, proposed by A., Garcke, and Gr\"un~\cite{AbelsGarckeGruen2}.
Here the chemical potential equation involves a term where a nonlocal operator as before is acting on the order parameter, instead of the Laplacian in the original model. More precisely, we consider
\begin{align}
 \partial_t (\rho \mathbf{v}) + \operatorname{div} ( \bv \otimes(\rho \bv + \tbj)) - \operatorname{div} (2 \eta(\tc) D \bv)
  + \nabla p  
  & = \mu \nabla \tc & \text{in } \, Q_T,  \label{eq:1} 
\\
 \operatorname{div} \, \bv &= 0& \text{in } \, Q_T,  \label{eq:2} 
\\
 \partial_t \tc + \bv \cdot \nabla \tc &= \operatorname{div}\left(m(\tc) \nabla \mu \right)& \text{in } \, Q_T, \label{eq:3} 
\\
 \mu = \Psi'(\tc)  &+\LL_\eps \tc & \text{in } \, Q_T, \label{eq:4} 
\end{align}
where $\rho=\rho(\tc):= \frac{\tilde{\rho}_1+\tilde{\rho}_2}2+ \frac{\tilde{\rho}_2-\tilde{\rho}_1}2\tc $, $\tbj = -\frac{\tilde{\rho}_2 - \tilde{\rho}_1}{2} m(\tc) \nabla \mu$, 
$Q_T=\Omega\times(0,T)$.

We assume that $\Omega \subset \mathbb{R}^d$, $d=2,3$, is a bounded domain with $C^2$-boundary. Here and in the following $\mathbf{v}$, $p$, and $\rho$ are the (mean) velocity, the pressure and the
density of the mixture of the two fluids, respectively. Furthermore $\tilde{\rho}_j$, $j=1,2$, are the specific densities of the unmixed fluids, 
 $\tc$ is the difference of the volume fractions of the two fluids, and $\mu$ is the chemical
potential related to $\tc$. 
Moreover,  ${D}\mathbf{v}= \frac12(\nabla \mathbf{v} + \nabla \mathbf{v}^T)$,
$\eta(\tc)>0$ is the viscosity of the fluid mixture, and $m(\tc)>0$ is a mobility coefficient.
Finally,
 $\LL_\eps$ is defined as in \eqref{eq:defnL}.
 Here the kernels $k_{\eps} \colon (\Rn\setminus\{0\})\to [0,\infty)$ satisfy the conditions (A1)-(A3) in Section~\ref{sec:prelim}. 

We add to our system  the boundary and initial conditions
\begin{alignat}{2}\label{eq:5}
 \bv|_{\partial \Omega} &= 0 &\qquad& \text{on}\ \partial\Omega\times (0,T), \\\label{eq:6}
 \partial_\no \mu|_{\partial \Omega} &= 0&& \text{on}\ \partial\Omega\times (0,T),  \\\label{eq:7}
 \left(\bv , \tc \right)|_{t=0} &= \left( \bv_0 , \tc_0 \right) &&\text{in}\ \Omega. 
\end{alignat}
We note that \eqref{eq:5} is the usual no-slip boundary condition for the velocity field and $\partial_\no \mu |_{\partial \Omega} = 0$ describes that there is no mass flux of the fluid components 
through the boundary. Furthermore, in dependence on how singular $k_\eps$ is at the origin, there will be an additional implicit boundary condition for $\tc$, which will be part of the weak formulation below.

There have recently been a lot of studies concerning existence of weak solutions and well-posedness of this model and its nonlocal variant. In particular, Frigeri~\cite{FrigeriNonlocalAGG} proved existence of weak solutions for $W^{1,1}$-kernels, which was extended by the authors to operators of regional fractional operator type in \cite{AbelsTerasawaNonlocalAGG}. We refer to \cite{GalGiroginiGrasselliPoiatti23} for a recent result for the local system and further literature.

For defining weak solutions of the system above, we introduce the following notation.
We introduce the standard Hilbert spaces for the Navier-Stokes
equations 
$$
L^2_\sigma(\Omega):=\overline{\mathcal{V}}^{L^2(\Omega)^d},\qquad
V_\sigma:=\overline{\mathcal{V}}^{H^1_0(\Omega)^d},\qquad\mathcal{V}:=\{\bv \in
C^\infty_0(\Omega)^d:\Div \bv=0\},
$$
and recall that these spaces can be characterized in the following way
$$
L^2_\sigma(\Omega):=\{\bv \in
L^2(\Omega)^d:\Div \bv=0,\:\:\bv \cdot \no |_{\partial\Omega}=0\},\quad
V_\sigma:=\{\bv \in H_0^1(\Omega)^d: \Div \bv =0\}.
$$
The norm and scalar product in $L^2_\sigma(\Omega)$ will be denoted again by
$\Vert\,\cdot\,\Vert$ and $(\cdot\,,\,\cdot)$, respectively, and
the space $V_\sigma$ is endowed with the scalar product
\begin{align*}
(\bv_1,\bv_2)_{V_\sigma}:=(\nabla\bv_1,\nabla\bv_2)=2\big(D\bv_1,D\bv_2\big)
\quad\text{for all }\bv_1,\bv_2\in V_\sigma. \nonumber
\end{align*}
We also introduce the Stokes operator $A$ with no-slip boundary condition. Recall that
$A\colon D(A)\subset L^2_\sigma(\Omega)\to L^2_\sigma(\Omega)$ is defined as $\,A:=-P\Delta$, with domain
$\,D(A)=H^2(\Omega)^d\cap V_\sigma$,
where $P\colon L^2(\Omega)^d\to L^2_\sigma(\Omega)$ is the Helmholtz projection.

For the equation above, we define its weak solution as follows. 
The definition is as in the author's contribution~\cite{AbelsTerasawaNonlocalAGG}.
\begin{definition} \label{def-weak}
  Let $\ve_0 \in L^2_\sigma(\Omega)$, $\tc_0 \in L^\infty(\Omega)$ with $|\tc_0|\leq 1$ almost everywhere, $T\in (0,\infty)$ and $\eps>0$ be given. Then $(\ve,\tc)$ is a weak solution of the nonlocal Navier-Stokes/Cahn-Hilliard system \eqref{eq:1}-\eqref{eq:7} if
  \begin{align*}
    \ve &\in BC_w([0,T]; L^2_\sigma(\Omega))\cap L^2(0,T; H^1_0(\Omega)^d),\\
    \tc &\in BC_w([0,T];L^2(\Omega))\cap L^\infty(0,T;V_\eps)\cap L^2(0,T;H^1(\Omega)),\\
    \mu &\in L^2(0,T; H^1(\Omega)), \quad f'(c)\in L^2(0,T;L^2(\Omega)),\\
    \partial_t (\rho \ve)|_{\mathcal{D}(A)}&\in L^{4/3} (0,T;\mathcal{D}(A)'), \quad \partial_t\tc \in L^2(0,T;(H^1(\Omega))'),
  \end{align*}
 $|\tc(x,t)|<1$ almost everywhere in $Q_T$, $\ve|_{t=0}= \ve_0$, $\tc|_{t=0}= \tc_0$ and the following holds true:
  \begin{enumerate}
  \item For every $\psi \in H^1(\Omega)$ and $\bpsi\in \mathcal{D}(A)$ and almost every $t\in (0,T)$ we have
    \begin{align*}
      \duality{\partial_t (\rho\ve)(t)}{\bpsi}_{\mathcal{D}(A)} - \int_\Omega ( (\ve+ \tbj) \otimes \rho\ve : D\bpsi \,dx &+ \int_\Omega 2\nu(\tc)D\ve:D\bpsi\, dx\\
      &= -\int_\Omega \tc \nabla \mu\cdot \bpsi \,dx,\\
      \duality{\mu(t)}{\psi}_{L^2(\Omega)} &= 2 \mathcal{E}_{\eps} (c(t), \psi) + \int_\Omega f'(c(t))\psi \,dx  \\     
      \duality{\partial_t \tc(t)}{\psi}_{H^1(\Omega)} + \int_\Omega m(\tc) \nabla \mu\cdot \nabla \psi &= \int_\Omega \ve \tc \cdot \nabla \psi\,dx,
    \end{align*}
    where $\tbj= -\tfrac{\tilde{\rho}_2 - \tilde{\rho}_1}{2}m(\tc) \nabla \mu$.
  \item The energy inequality
    \begin{equation}
      \label{eq:Energy}
      E_{tot, \ep}(\ve(t),\tc(t))+ \int_0^t\int_\Omega (2\nu(\tc)|D\ve|^2+ m(\tc)|\nabla \mu|^2) \,dx\, d\tau \leq E_{tot, \ep}(\ve(0),\tc(0))
    \end{equation}
    holds true for all $t \in [0,T]$, where
    \begin{align*}
            E_{tot, \ep}(\ve,\tc)&:= \frac12\int_\Omega \rho(\tc)|\ve|^2\, dx + E_\ep(\tc),\\
      E_\ep(\tc) &:= \frac14 \int_\Omega \int_\Omega k_\ep(x-y) (\tc(x)-\tc(y))^2 \,dx\, dy + \int_\Omega f(\tc(x))\, dx.
    \end{align*}
  \end{enumerate}
\end{definition}
\begin{remark} 
The existence of weak solutions of the nonlocal Navier-Stokes/Cahn-Hilliard equation in the sense of Definition \ref{def-weak} in the case when $ k_{\ep} \in W^{1, 1}(\mathbb{R}^d)$ follows from Frigeri
\cite{FrigeriNonlocalAGG} and in the case of the first example of the kernel in Remark 2.1 follows from \cite{AbelsTerasawaNonlocalAGG}. The definition of weak solutions in \cite{AbelsTerasawaNonlocalAGG} is different from Definition \ref{def-weak}, but it can be shown that it is equivalent in a standard manner. For details, we refer to Boyer and Fabrie \cite{BoyerFabrie}.
\end{remark}

\begin{theorem}
For any $\eps \in (0, 1)$, let $ \ve_{0, \eps} \in L^2_\sigma(\Omega) $ and $ \tc_{0, \eps} \in L^2(\Omega)$ with $ | \tc_{0, \eps} | \leq 1$ a.e., 
$ \frac{1}{|\Omega|}\int_{\Omega} \tc_{0, \eps}(x) = m_{\Omega} $ for all $ \eps \in (0, 1) $ and some $ m_{\Omega} \in (-1, 1)$. Moreover, we assume that there are 
$ \ve_0 \in L^2_{\sigma}(\Omega) $ and $\tc_0 \in H^1(\Omega)$ such that 
$ \ve_{0, \eps} \rightarrow \ve_0 $ in $ L^2_{\sigma}(\Omega) $, $\tc_{0, \eps} \rightarrow \tc_0$ in $ L^2(\Omega)$, and $ E_{\eps, tot}(\ve_{0, \eps}, \tc_{0, \eps}) \rightarrow
E_{tot}(\ve_0, \tc_0)$  as $ \eps \rightarrow 0+$, where 
\begin{align*}
E_{tot}(\ve, \tc) &:= \frac12 \int_{\Omega} \rho(\tc) |\ve|^2\,dx + E_0(\tc),\\
E_0(\tc) &:= \frac12 \int_{\Omega} (A \nabla \tc(x))\cdot \nabla \tc(x)\,dx + \int_{\Omega} f(\tc(x)) \,dx.
\end{align*}
If $ (\ve_{\eps}, \tc_{\eps})_{\eps\in (0,1)} $  
are weak solutions of the nonlocal Navier-Stokes/Cahn-Hilliard system with initial values $(\ve_{0, \eps}, \tc_{0, \eps})_{\eps\in (0,1)}$, then
\begin{alignat}{2}\label{eq:Conv1}
\mathbf{v}_{\eps} &\rightharpoonup^\ast \mathbf{v}&\qquad &\text{ in }L^{\infty}(0, T;L^2_{\sigma}(\Omega)), \\\label{eq:Conv2}
\mathbf{v}_{\eps} &\rightharpoonup \mathbf{v}&&\text{ in }L^2(0, T; H^1_0 (\Omega)^d), \\\label{eq:Conv3}
\mu_{\eps} &\rightharpoonup \mu && \text{ in }L^2(0, T;H^1(\Omega)),\\\label{eq:Conv4} 
\mathbf{v}_{\eps} &\rightarrow \mathbf{v} && \text{ in }L^2(0, T; L^2(\Omega)) \text{ and almost everywhere}, \\\label{eq:Conv5}
\tc_{\eps} &\rightarrow \tc && \text{ in  }C([0,T];L^2(\Omega)) \text{ and almost everywhere}
\end{alignat}
for a suitable subsequence $\eps =\eps_k\to_{k\to\infty} 0$,
where  $ (\ve,\tc) $ is a weak solution of the local Navier-Stokes/Cahn-Hilliard system in the sense that
\begin{align*}
  \ve &\in BC_w([0,T];L^2_\sigma(\Omega))\cap L^2(0,T;H^1_0(\Omega)^d),\\
  \tc & \in C([0,T]; L^2(\Omega))\cap BC_w([0,T]; H^1(\Omega))\cap L^2(0,T;H^2(\Omega)), f'(\tc) \in L^2(0,T;L^2(\Omega)),\\
  \mu & \in L^2(0,T; H^1(\Omega)),
\end{align*}
 $|\tc(x,t)|<1$ almost everywhere in $Q_T$, $\ve|_{t=0}= \ve_0$, $\tc|_{t=0}= \tc_0$ and the following holds true:
  \begin{enumerate}
  \item For every $\psi \in H^1(\Omega)$ and $\bpsi\in \mathcal{D}(A)$ and almost every $t\in (0,T)$ we have
    \begin{align*}
      \duality{\partial_t (\rho\ve)(t)}{\bpsi}_{\mathcal{D}(A)} - \int_\Omega ( (\ve+ \tbj) \otimes \rho\ve : D\bpsi \,dx &+ \int_\Omega 2\nu(\tc)D\ve:D\bpsi\, dx\\
      &= -\int_\Omega \tc \nabla \mu\cdot \bpsi \,dx, \\  
      \duality{\partial_t \tc(t)}{\psi}_{H^1(\Omega)} + \int_\Omega m(\tc) \nabla \mu\cdot \nabla \psi &= \int_\Omega \ve \tc \cdot \nabla \psi\,dx, 
    \end{align*}
    where $\tbj= -\tfrac{\tilde{\rho}_2 - \tilde{\rho}_1}{2}m(\tc) \nabla \mu$, and
    \begin{alignat*}{2}
               \mu& = -\Div (A\nabla \tc) +f'(c) &\qquad& \text{in }\Omega\times (0,T), \\
           \no\cdot A\nabla \tc|_{\partial\Omega}&=0&\qquad &\text{on }\partial\Omega\times (0,T).
    \end{alignat*}
  \item The energy inequality
    \begin{equation}
      \label{eq:Energy}
      E_{tot, 0}(\ve(t),\tc(t))+ \int_0^t\int_\Omega (2\nu(\tc)|D\ve|^2+ m(\tc)|\nabla \mu|^2) \,dx\, d\tau \leq E_{tot, 0}(\ve(0),\tc(0))
    \end{equation}
    holds true for all $t \in [0,T]$.
  \end{enumerate}
\end{theorem}

\begin{proof}
The proof is along the same lines as that of the main theorem in \cite{AbelsTerasawaNonlocalToLocal}. Here one has to deal with the convergence for the more general kernels, using the results of Section~\ref{sec:NonlocalLimit}. First we prove that $\tc_{\eps} \rightarrow_{\eps\to 0} \tc$ strongly in $C([0,T];L^2(\Omega))$ and almost everywhere
employing Lemma~\ref{lem:3.3} and the boundedness of $ \ve_{\eps} \in L^2(0, T; L^2_\sigma(\Omega)) $ derived from the energy inequality of the nonlocal Navier-Stokes/Cahn-Hilliard system. 
After deriving it, by similar arguments as in \cite{AbelsTerasawaNonlocalToLocal}, we can show that $ \ve_{\eps} \rightarrow_{\eps\to 0} \ve $ in $ L^2(0, T; L^2(\Omega))^d$ if we take some subsequence of $ (\ve_{\eps})_{\eps\in (0,1)} $. The fact that $\tc$ satisfies the weak formulation can be shown using Theorem \ref{maintheorem}. The rest is similar to \cite{AbelsTerasawaNonlocalToLocal} and we omit it.
\end{proof}

\begin{remark}
From the physical point of view, the case when the kernel $k_{\ep} $ is radially symmetric is more natural taking into account isotropy of the fluids. In that case, the limit equation is also isotropic, i.e., $A = c I$ where $c>0$ is a positive constant and $I$ is an identity matrix.
\end{remark}





\section*{Acknowledgments}
This work was supported by the University of Regensburg Foundation Hans Vielberth. The second author was supported by JSPS KAKENHI Grant-in-Aid for Young Scientists (B) \#JP17K17804. He would also like to thank Professor Mitsuru Sugimoto for his support by JSPS KAKENHI Grant-in-Aid for Scientific Research (A) \#JP22H00098. This work was also supported by the Research Institute for Mathematical Sciences, an International Joint Usage/Research Center located in Kyoto University. All these supports are gratefully acknowledged.
\bibliographystyle{abbrv}

\bibliography{Bibliography} 

\end{document}